\documentclass[12pt,reqno]{amsart}

\usepackage{latexsym}
\usepackage{amssymb}

\renewcommand{\Re}{{\operatorname{Re}\,}}
\renewcommand{\Im}{{\operatorname{Im}\,}}

\newcommand{\s}{{\mathbf S}}

\renewcommand{\epsilon}{\varepsilon}

\newcommand{\var}{{\operatorname{Var}}}

\newcommand{\poly}{{\operatorname{Poly}}}

\newcommand{\sm}{\smallsetminus}
\newcommand{\szego}{Szeg\H{o} }

\newcommand{\inv}{^{-1}}
\newcommand{\kahler}{K\"ahler }

\newcommand{\wt}{\widetilde}
\newcommand{\wh}{\widehat}
\newcommand{\PP}{{\mathbb P}}

\newcommand{\R}{{\mathbb R}}
\newcommand{\C}{{\mathbb C}}

\newcommand{\Z}{{\mathbb Z}}

\newcommand{\CP}{\C\PP}
\renewcommand{\d}{\partial}
\newcommand{\dbar}{\bar\partial}
\newcommand{\ddbar}{\partial\dbar}

\newcommand{\E}{{\mathbf E}}

\newcommand{\half}{{\textstyle \frac 12}}
\newcommand{\vol}{{\operatorname{Vol}}}

\newcommand{\SU}{{\operatorname{SU}}}
\newcommand{\FS}{{{\operatorname{FS}}}}

\renewcommand{\phi}{\varphi}

\newcommand{\ccal}{\mathcal{C}}
\newcommand{\dcal}{\mathcal{D}}

\newcommand{\lcal}{\mathcal{L}}

\newcommand{\ocal}{\mathcal{O}}
\newcommand{\pcal}{\mathcal{P}}

\newcommand{\scal}{\mathcal{S}}

\newcommand{\al}{\alpha}

\newcommand{\ga}{\gamma}

\newcommand{\la}{\lambda}
\newcommand{\ep}{\varepsilon}
\newcommand{\de}{\delta}

\newcommand{\om}{\omega}
\newcommand{\Om}{\Omega}

\newtheorem{theo}{{\sc Theorem}}[section]

\newtheorem{cor}[theo]{{\sc Corollary}}

\newtheorem{lem}[theo]{{\sc Lemma}}
\newtheorem{prop}[theo]{{\sc Proposition}}

\newenvironment{rem}{\medskip\noindent{\it Remark:\/} }{\medskip}

\title{Convergence of random zeros on complex manifolds}

\author{Bernard Shiffman}
\address{Department of Mathematics, Johns Hopkins University, Baltimore, MD
21218, USA} \email{shiffman@math.jhu.edu}

\thanks{Research partially supported by NSF grant
DMS-0600982.}

\begin{document}

\begin{abstract}  We show that the zeros of random sequences of
Gaussian systems of polynomials of increasing degree almost surely converge to
the expected limit distribution under very general hypotheses.  In particular,
the normalized distribution of zeros of systems of $m$ polynomials of degree
$N$, orthonormalized on a regular compact set $K\subset\C^m$, almost surely
converge to the equilibrium measure on $K$ as $N\to\infty$.
\end{abstract}

\maketitle
 \section{Introduction}

The central theme of  this paper is the almost sure convergence to an equilibrium distribution of zeros of 
random sequences of holomorphic zero sets.  We work with simultaneous
zeros of random polynomials on $\C^m$ or, more generally, zeros of
random sections of powers of a holomorphic line bundle $L\to M$
over a compact \kahler manifold.  To review some history, the asymptotic 
properties of the zeros of random real polynomials were studied by Kac \cite
{Kac}   in 1949; a few years later, Hammersley \cite{Ham} investigated the 
zeros of the complexification of the Kac ensembles.  While the zeros of the 
Hammersley ensembles tend to accumulate on the unit circle in $\C$, the 
distribution of zeros is uniform (with respect to the Fubini-Study measure on 
$\CP^1$)  for the ``$\SU(2)$ polynomials'' studied in the physics literature 
(e.g., \cite{BBL,FH,Han,NV}). There has been a recent interest in the 
statistical properties of zeros and simultaneous zeros of random functions of 
several variables.  For example,  statistics on zeros and simultaneous zeros 
of random polynomials of several real variables were given in  \cite
{EK,Ro,SS,Ws}.  Results on zeros of random polynomials of several complex 
variables as well as of random holomorphic sections of line bundles can be 
found in  
\cite{Be1,BSZuniv,BSZsusy,BS,EK,SZ,SZnewton, SZvar,Zr} and elsewhere.
 
 In joint work with Zelditch
\cite{SZ} in 1999, we showed that if $L$ is a positive Hermitian line
bundle, the normalized zero currents $\frac 1N Z_{s_N}$ of a
random sequence $s_N\in H^0(M,L^N)$ of holomorphic sections of
increasing powers of $L$ almost surely converge to the curvature
form of $L$. This result was derived as a consequence of an
asymptotic expansion for the expected values $\E\big(\frac 1N
Z_{s_N}\big)$ of these zero currents together with an elementary
variance estimate. Furthermore, as a consequence of the sharp
variance asymptotics in a recent paper with Zelditch \cite{SZvar},
the normalized expected zero currents $\frac 1{N^k}
Z_{s_N^1,\dots,s_N^k}$ of $k$ independent random sections almost
surely converge to a uniform distribution (for $1\le k\le \dim M$). On the other
hand,  Bloom showed in \cite{Bl} that for random sequences of
polynomials of increasing degree, orthonormalized with respect to
certain measures on a compact set $K\subset \C$, the normalized
zero distributions converge almost surely to the equilibrium
measure on $K$.

In this paper, we  show that for all sequences of ensembles of
random sections of increasing powers of line bundles (e.g., random
polynomials of increasing degree), whenever the expected
normalized zero currents converge to a limit current, the
convergence holds almost surely for random sequences.  The only
condition imposed on the sequence of ensembles is that the probability measures are (complex)
Gaussian.

Our results are stated in terms of the currents of integration
over zero sets, which we call {\it zero currents.\/} For a system
$\s_N = (s_N^1, \dots, s_N^k)$ of $k$ holomorphic sections
$s_N^j\in H^0(M, L^N)$, $j=1,\dots,k$ (where $1\le k\le m=\dim
M$), we let
$$|Z_{\s_N}|:=\{z\in M:s_N^1(z) = \cdots = s_N^k(z)
= 0\}$$ denote its zero set, and we consider the current of
integration $Z_{{\s}_N}\in \dcal'^{k,k}(M)$ defined by

$$\left( Z_{{\s}_N},\phi\right)=\int_{|Z_{\s_N}|}
\phi\;,\qquad \phi\in \dcal^{m-k,m-k}_\R(M)\;,$$ whenever the zero
set of ${\s_N}$ is a codimension $k$ subvariety without multiplicity. 
(For $L^N$ base point free, $|Z_{\s_N}|$ is almost surely a smooth
codimension $k$ subvariety without multiplicity.) We recall that $\dcal^{j,j}_\R(M)$ 
denotes the space of real $\ccal^\infty$ forms of bidegree $(j,j)$ on $M$.

Our convergence result (Corollary \ref{convergence}) is a
consequence of the following  variance estimate:

\begin{theo}\label{var-intro} Let $L\to (M,\om)$ be an ample
holomorphic line bundle  over a compact \kahler manifold of
dimension $m$, and let $\scal$ be a linear subspace of the space $
H^0(M,L)$ of holomorphic sections of $L$. Suppose that $\scal$ has
a Gaussian probability measure. Let $1\le k\le m$ and let $\phi\in
\dcal^{m-k,m-k}_\R(M\sm B)$, where $B$ is the base point set of
$\scal$. Then the standard deviation of the zero statistics of $k$
independent random sections $s_1,\dots,s_k$ of $\scal$ satisfies
the bound
$$\sqrt{\var ( Z_{s_1,\dots,s_k},\phi)} \le C_m\,
\|\ddbar \phi\|_{\infty}\int_M \om^{m-k+1}\wedge c_1(L)^{k-1}\;,$$
where the constant
$C_m$ depends only on the dimension $m$ of $M$.
\end{theo}

The base point set of $\scal$ is the set of points $z\in M$ where
$s(z)=0$ for all $s\in\scal$.  In \S \ref{s-variance}, we prove a
slightly more general variance bound (Theorem \ref{mainvar}).

The key point is that the variance bound involves the $(k-1)$-th
power of $ c_1(L)$ instead of the $k$-th power. It follows that
the standard deviations of simultaneous zeros of random sections
of  the $N$-th tensor powers $L^N$ of $L$ grow at a lower rate than the expected 
values. To be
precise, given $k$
sections $s_N^1,\dots,s_N^k$ of $L^N$, we define the {\it normalized zero current\/}
\begin{equation}\wt Z_{s_N^1,\dots,s_N^k}:=\frac 1 {N^k}
Z_{s_N^1,\dots,s_N^k}\;.\end{equation} We then have the following
asymptotic variance bound:

\begin{theo}\label{varN} Let $L\to M$ be a holomorphic
line bundle over a projective algebraic manifold of dimension $m$ and let  $\phi\in
\dcal^{m-k,m-k}_\R(U)$, where $U$ is an open subset of $M$. Suppose that  we are given 
subspaces $\scal_N^j\subset
H^0(M,L^N)$ endowed with arbitrary complex Gaussian probability
measures $\ga_N^j$, such that $\scal^j_N$ has no base points in $U$, for $1\le j \le 
k,\ N \ge 1$.  Then
for independent random sections $s_N^j\in \scal^j_N$,
$${\var (\wt  Z_{s_N^1,\dots,s_N^k},\phi)} \le O\left(\frac 1{N^2}\right)\,.$$
\end{theo}

If $L$ is ample,  the conclusion of Theorem \ref{varN} follows
immediately from Theorem \ref{var-intro}.  For the general case,
the result follows from a modified version (Proposition \ref{var-var}) of Theorem \ref{var-intro}.

The fundamental case considered in \cite{SZ, SZvar} is where $L$
is an ample line bundle and $\scal_N^j=H^0(M,L^N)$.  Then  $L$ has
a Hermitian metric $h$ with positive curvature, and we give $M$
the \kahler form $\om=\pi c_1(L,h)$, where $c_1(L,h)$ is the Chern
form (see \eqref{chern}). The Hermitian metric $h$ on $L$ and the \kahler
form $\om$ induce Hermitian inner products on the spaces $H^0(M,L^N)$:
\begin{equation}\label{inner}\langle s_N, \bar s'_N \rangle = \int_M h^N(s_N,
s'_N)\,\frac 1{m!}\om^m\;,\qquad s_N,s_N' \in
H^0(M,L^N)\,\;,\end{equation} where $h^N$ denotes the induced metric on $L^N$. These inner products in turn induce 
Gaussian probability measures on the corresponding spaces (see \eqref{gaussian}). It was shown in
\cite{SZ} that  in this case, \begin{equation}\label{asymptexp}\E\big(\wt
Z_{s^1_N,\dots,s^k_N},\phi\big)=\int_M\om^k\wedge\phi+O\left(\frac 1N\right)
\;,\end{equation} where $\E(Y)$ denotes the expected value of a random variable $Y$.
For the fundamental case, we further have the  sharp variance bound from
\cite{SZvar1}:
\begin{equation}\label{sharper}\var\big(\wt Z_{s^1_N,\dots,s^k_N},\phi\big)
\le O\left(\frac 1{N^{m+2}}\right)\;.\end{equation}   In fact when  $k=1$,
we have the precise formula
\begin{equation}\label{precise}\textstyle\var\big(\wt Z_{s^N},\phi\big) = N^{-m-2}
\left[ \frac{\pi^{m-2}\,\zeta(m+2)}{4}\,\|\ddbar\phi\|_{L^2}^2
 +O(N^{-\frac 12 +\ep})\right]\;.\end{equation}  The variance formula \eqref{precise} 
 was previously obtained for zeros of polynomials in one variable (the $\SU(2)$ ensemble) by 
 Sodin and Tsirelson \cite{ST}.

 We point out here that the weaker bound of Theorem \ref{varN} holds for an
arbitrary sequence of Hermitian inner products on  arbitrary subspaces 
$\scal^j_N \subset H^0(M,L^N)$, whereas the sharp bound \eqref{sharper} is a 
consequence of the off-diagonal asymptotics of the \szego kernels for the full spaces $H^0(M,L^N)$ with the 
inner products \eqref{inner}. (See \cite{BSZuniv,SZvar} for discussions of the \szego kernel asymptotics.)  

Theorem \ref{varN} implies the following
general result on almost sure convergence to the average of sequences of
zeros of i.i.d.\
$k$-tuples of sections in $\scal_N$:

\begin{cor} \label{convergence}  Let $L\to M$ be a
holomorphic line bundle over a projective algebraic manifold. Suppose that  $\ga_N$ is 
a Gaussian probability
measure on  a subspace $\scal_{N}$ of $H^0(M,L^N)$, for
$N=1,2,3,\dots$, and let $\ga=\prod_{N=1}^\infty \ga_N^k$ denote
the product measure on $\scal_\infty:= \prod_{N=1}^\infty
(\scal_N)^k$. Let $U$ be an open subset of $M$ such that $\scal_N$
has no base points on $U$, for all $N$.

Suppose that the expected normalized zero currents
$$\E_{\ga_N^k}\left(\wt Z_{s_N^1,\dots,s_N^k}|U\right)$$ converge
weakly in $\dcal'^{k,k}(U)$ to a current $\Psi\in\dcal'^{k,k}(U)$.
Then for $\ga$-almost all sequences $\{(s_N^1,\dots,s_N^k)\}_{N=
1}^\infty\in{\mathcal S}_\infty$, $$ \wt Z_{s_N^1,\dots,s_N^k}|U
\to\Psi \quad\mbox{ weak}^*$$ (in the sense of measures); i.e., for almost all sequences,
$$\lim_{N\to\infty}\left(\frac{1}{N^k} Z_{s_N^1,\dots,s_N^k},\phi\right)=
\int_M\Psi \wedge\phi$$ for all continuous $(\dim M -k,\dim M -k)$ forms
$\phi$ with support contained in $U$.
\end{cor}

The proof of Corollary \ref{convergence} is given in \S \ref{AE}.
Applying Corollary  \ref{convergence} to the full ensembles $\scal^j_N=H^0(M,L^N)$ with the inner product \eqref{inner}, we conclude from \eqref{asymptexp} that the simultaneous zeros of random
sequences $\{(s^1_N,\dots,s^k_N)\}$ are almost always
asymptotically uniform; i.e., 
$\wt Z_{s^1_N,\dots,s^k_N}
\to\om^k$ almost surely, as noted in \cite{SZvar} using the sharp variance bound \eqref{sharper} (and in \cite{SZ} for the case $k=1$).

We now mention some new applications of Corollary \ref{convergence}.
The first application is to the result given in joint work with Bloom \cite{BS} on zeros of
random polynomial systems orthonormalized on compact sets in
$\C^m$:

\begin{theo}\label{general} Let $\mu$ be a Borel probability
measure on a regular compact set $K\subset
\C^m$, and suppose that $(K,\mu)$ satisfies a Bernstein-Markov
inequality. Let $1\le k\le m$, and let $(\pcal_N^k, \ga_N^k) $
denote the ensemble of $k$-tuples of i.i.d.\ Gaussian random
polynomials of degree $\le N$ with the Gaussian measure $d\ga_N$
induced by $L^2(\mu)$. Then for almost all sequences of $k$-tuples of 
polynomials $\{(f^1_N,\dots,f^k_N)\}\in \prod_{N=1}^\infty \pcal_N^k$,
$$\wt Z_{f^1_N,\dots,f^k_N} \to \left(\frac i{\pi}\ddbar
V_K\right)^k \quad \mbox{weak}^*\,,$$
where $V_K$ is the pluricomplex Green function of $K$ with pole at
infinity.  In particular, for $k=m$,
$$\wt Z_{f^1_N,\dots,f^m_N} \to \mu_{eq}(K):=\left(\frac i{\pi}\ddbar
V_K\right)^m \quad \mbox{weak}^*\quad a.s.\;.$$
\end{theo}

The one-variable case of Theorem \ref{general} was given in \cite{Bl}, generalizing a 
result in  \cite{SZequil}.
The pluricomplex Green function in the theorem is given by\begin{equation}
V_K(z):=\sup\{u (z) \in \lcal:\ u\le 0\ \ {\rm on}\ \ K\}\;,
\end{equation} where \begin{equation}
\lcal:=\{u\in \mbox{PSH} (\C^m): u(z)\le \log^+
\|z\| +O(1)\}\;.
\end{equation}
If $\mu$ is a probability measure on a compact set
$K\subset\C^m$, one says that $(K,\mu)$ satisfies a
Bernstein-Markov inequality if for all $\ep>0$, there is a
positive constant $C=C(\ep)$ such that
\begin{equation}\label{BM}
\|p\|_K\le C e^{\ep\deg(p)} \|p\|_{L^2(\mu)}\,,
\end{equation}
for all  polynomials $p$. The measure $\mu_{eq}(K)$ is called the 
equilibrium measure of $K$; it is supported on the Silov boundary of $K$, and $(K,\mu_{eq}(K))$ satisfies a
Bernstein-Markov inequality (for $K$ regular).

In \cite{BS}, we showed that that the expected values of the normalized
zero currents of Theorem \ref{general} satisfy the asymptotics:
\begin{equation}\label{equilib}\E \left(\wt Z_{f^1_N,\dots,f^k_N}\right)\to \left(\frac i{\pi}\ddbar V_K\right)^k \quad \mbox{weak}^*\;.\end{equation} Theorem \ref{general} follows from  Corollary \ref{convergence} and \eqref{equilib}.
 
We remark that a generalization of \eqref{equilib} with weights was recently given by Bloom 
\cite{Bl2}, answering a question posed in \cite{SZequil}. To state Bloom's result, we let $w:K
\to [0,+\infty)$ be a continuous weight (such that the set $\{w>0\}$ is non-pluripolar), and we 
give each space $\pcal_N$ of polynomials of degree $\le N$ the Gaussian measure induced by $L^2
(w^{2N}d\mu)$.  We let $\phi=-\log w$ and define the ``weighted pluricomplex Green function" 
$$V_{K,\phi}(z):=\sup\{u (z)\in \lcal:\ u\le \phi\ \ {\rm on}\ \ K\}\,.$$
If $(K,\mu)$ satisfies a ``weighted Bernstein-Markov inequality" (replace $p$ with $w^Np$ in 
\eqref{BM}), one then has the asymptotics
\begin{equation}\label{Bloom} \E_{w^N} \left(\wt Z_{f^1_N,\dots,f^k_N}\right)\to \left(\frac i
{\pi}\ddbar V_{K,\phi}\right)^k \quad \mbox{weak}^*\;,\end{equation}  where $\E_{w^N}$ denotes 
the expected value for the weighted ensemble \cite[Th.~2.1]{Bl2}.  It then follows as before 
from Corollary \ref{convergence} that
$$\wt Z_{f^1_N,\dots,f^k_N} \to \left(\frac i{\pi}\ddbar
V_{K,\phi}\right)^k \quad \mbox{weak}^*\,,\quad  \mbox{a.s.}\;.$$

Our next application is to systems of random
polynomials  with  fixed Newton polytopes  as discussed in \cite{SZnewton}. Given
a convex integral polytope $P\subset [0,+\infty)^m$, we  denote by
$\poly(P)$ the space of polynomials
$$f(z_1, \dots, z_m) = \sum_{\alpha \in P\cap\Z^m} c_{\alpha} z_1^{\al_1}\cdots z_m^{\al_m}$$
with Newton polytope contained in $P$. It is a subspace of $H^0(\CP^m,
{\mathcal O}(p))$, the space of all homogeneous polynomials of degree $p$,
where $p$ is  the maximal degree of polynomials in $\poly(P)$. The
$\SU(m+1)$-invariant inner product on $H^0(\CP^m,{\mathcal O}(p))$
then restricts to $\poly(P)$ to define an inner product and
Gaussian measure there. It is the conditional Gaussian measure on
polynomials with the condition of having Newton polytope $P$. In joint work with Zelditch
\cite{SZnewton}, we studied the asymptotic statistical patterns of
zeros of polynomials in $\poly(NP)$, where $NP$ denotes the dilate
of $P$ by $N$.

We apply Corollary \ref{convergence} with $L= \ocal(1)\to
M=\CP^m$ and
 $\scal_N=\poly(NP)$ with the conditional Gaussian measure described above.
With this choice of ensembles, the expected zero current is not uniformly
distributed over $\CP^m$. Instead, it was shown in \cite{SZnewton} that for
each integral polytope
$P$,
 there is associated  a (discontinuous, piecewise smooth)  $(1,1)$-form
$\psi_P$ on $(\C^*)^m$ (where $\C^*=\C\sm\{0\}$) so that
\begin{equation} \label{MEANPOL} \E\left(\wt Z_{f_N}|(\C^*)^m\right)\to
\psi_P \quad \mbox{and hence} \quad \E\left(\wt
Z_{f^1_N,\dots,f^k_N} |(\C^*)^m\right)\to
\psi_P^k\,, \quad \mbox{for }\ 1\le k\le m.
\end{equation}
By Corollary
\ref{convergence}, we then have:

\begin{theo} \label{zerotheorem} For almost all sequences
$\{(f^1_N,\dots,f^k_N)\}\in\poly(NP)^k$, $N=1,2,3\dots$, $$ \wt
Z_{f^1_N,\dots,f^k_N}|(\C^*)^m \to\psi_P^k \quad
\mbox{weak}^*\;.$$ \end{theo}

In fact, to each polytope $P$ there is associated an {\it allowed
region\/} ${\mathcal A}_P \subset (\C^*)^m$ where
$\psi_P=p\,\om_\FS$ (where $\om_\FS$ denotes the Fubini-Study form on $\CP^m$), and hence the zeros of random sections of
$\poly(NP)$ tend to be equidistributed on ${\mathcal A}_P$, for $N$ large. On the complementary {\it forbidden region\/},
$\psi_P^m=0$ and hence  a random system of
$m$ polynomials with Newton polytope $NP$ has, on average,
few zeros in the forbidden region, for $N$ large.
It follows from Theorem \ref{zerotheorem} that   sequences of
simultaneous  zeros of systems of random polynomials
$f_N^1,\dots,f_N^m$ in $\poly(NP)$ will almost surely become concentrated in the allowed
region ${\mathcal A}_P$ and be uniformly distributed
there as $N\to\infty$.  

R. Berman \cite {Be2} recently gave an extension to non-positively curved line bundles  of the 
\szego kernel asymptotics of \cite{Ca,Ti,Ze} on which \eqref{asymptexp} is 
based. These asymptotics lead to similar convergence results for random zeros. 
To state Berman's result, we let $(L,h)\to (M,\om)$ be an ample Hermitian line 
bundle over a compact \kahler manifold.  Although $L$ is assumed to be ample, 
we do not assume that the metric $h$ has positive curvature. We give $H^0
(M,L^N)$ the inner product \eqref{inner} and the induced \szego kernel $\Pi_N$ 
and Gaussian probability measure (see \eqref{gaussian}--\eqref{sdef}).
 We let $\lcal_{(X,L)}$ denote the class of all (possibly singular) metrics on 
$L$ with positive curvature form, and we define the  {\it equilibrium metric\/} 
$h_e$ on $L$ by
\begin{equation}\label{equimetric} h_e:= \inf\{\tilde h\in  \lcal_{(X,L)} : 
\tilde h \ge h\}\;.\end{equation}
Choosing a local nonvanishing section $e_L$ of $L$, we write
$\phi_e= -\log |e_L|^2_{h_e}$, which is plurisubharmonic.  Berman  showed \cite
[Th.~2.3]{Be2} that $\phi_e$ is $\ccal^{1,1}$ and that the ``equilibrium 
measure" 
$$\mu_{eq}(h):=\Big(\frac i{2\pi} \ddbar\phi_e\Big)^m = c_1(L,h_e)^m$$ is 
absolutely continuous, i.e., is given by pointwise multiplication of the Chern 
forms.  Berman  then showed \cite[Th.~3.6]{Be2} that 
\begin{equation}\label{Berman} \frac 1N \log \Pi_N(z,z)\to 
\log \frac{ h(z)}{h_e(z)} \quad \mbox{uniformly},\end{equation} and hence 
\begin{equation}\label{Berman2} \E\big(\wt Z_{s^1_N,\dots,s^m_N}\big) \to \mu_
{eq}(h)\quad \mbox{weak}^*\;.
\end{equation}  Thus it follows from \eqref{Berman2} and Corollary \ref
{convergence} that 
\begin{equation}\label{Berman3} \wt Z_{s^1_N,\dots,s^m_N} \to \mu_{eq}(h)\quad 
\mbox{weak}^*\quad \mbox{a.s.}\;.
\end{equation}
Similar results hold for equilibrium measures on pseudoconcave domains in 
compact \kahler manifolds (see \cite{Be1}).

\section{Expected distribution of zeros and \szego kernels}\label{EDZ}

In this section, we review the formulas from \cite{SZ,SZvar} for
the expected current of integration over the zero set of $ k
\leq m$ i.i.d.\ Gaussian random sections of a holomorphic line bundle.

Let $(L,h)$ be a Hermitian holomorphic line bundle over a complex
manifold
$M$ and let $\scal$ be  a finite-dimensional subspace of $H^0(M,L)$ with a 
Hermitian inner product. The inner product on $\scal$ induces the complex Gaussian probability
measure
\begin{equation}\label{gaussian}d\gamma(s)=\frac{1}{\pi^m}e^
{-|c|^2}dc\,,\qquad s=\sum_{j=1}^{n}c_jS_j\,,\end{equation} on
$\scal$, where $\{S_j\}$ is an orthonormal basis for
$\scal$ and $dc$ is $2n$-dimensional Lebesgue measure. This
Gaussian is characterized by the property that the $2n$ real
variables $\Re c_j, \Im c_j$\break ($j=1,\dots,n$) are independent
random variables with mean 0 and variance $\half$; equivalently,
$$\E c_j = 0,\quad \E c_j c_k = 0,\quad  \E c_j \bar c_k =
\de_{jk}\,.$$

To state the explicit formula for the expected distribution of
zero divisors, we let
\begin{equation}\label{sdef}\Pi_\scal(z,z) = \sum_{j=1}^n
|S_j(z)|_h^2\;,\qquad z\in M\;,\end{equation} denote the
{\it \szego kernel\/} for
$\scal$ on the diagonal.

\begin{rem} The \szego kernel for the fundamental case $\scal=H^0(M,
L)$ (with the inner product \eqref{inner}) is given as follows: we let $X{\buildrel {\pi}\over \to} M$
denote the circle bundle of unit vectors in the dual bundle
$L\inv\to M$, and we identify sections $s\in \scal$ with functions
$\hat s$ in the space $\wh\scal$ of $\ccal^\infty$ functions on
$X$ such that $\dbar_b\hat s=0$ and $\hat s(e^{i\theta}x)=
e^{i\theta}\hat s(x)$. The {\it \szego projector\/} is the
orthogonal projector $\Pi:\lcal^2(X)\to\wh\scal$, which is given
by the {\it \szego kernel}
$$\Pi(x,y)=\sum_{j=1}^n \wh S_j(x)\overline{\wh S_j(y)}\qquad (x,y\in X)\;.$$
On the diagonal, we may write $\Pi(z,z)=\Pi(x,x)$, where
$\pi(x)=z$; then $\Pi(z,z)=\Pi_\scal(z,z)$ as defined by
(\ref{sdef}).  For details, see \cite{SZ}.
\end{rem}

\medskip

We now consider a local holomorphic frame  $e_L$  over
a trivializing chart $U$, and we write $S_j = f_j e_L$ over
$U$. Any  section $s\in\scal$ may then be written as
$$s = \langle c, F \rangle e_L\;, \quad \mbox{where\ \ \ }
F=(f_1,\dots,f_n)\;,\quad\langle c,F \rangle = \sum_{j = 1}^n c_j
f_j\;.$$ If  $s = f e_L$,  its Hermitian norm is given by
$|s(z)|_h = a(z)^{-\half}|f(z)|$ where $a(z) =
|e_L(z)|_h^{-2}$. The {\it Chern form\/} $c_1(L,h)$ of $L$ is
given locally by
\begin{equation}\label{chern}c_1(L,h)=\frac{\sqrt{-1}}{2 \pi}\d\dbar\log a\;.\end
{equation}
The current of integration over the zeros of $s = \langle c,F
\rangle\, e_L$ is then given locally by the
Poincar\'e-Lelong formula:
\begin{equation} Z_s =
\frac{\sqrt{-1}}{ \pi } \partial \bar{\partial}\log | \langle c,F
\rangle|\;. \label{Zs} \end{equation} 

We now recall the formula for the expected zero divisor for the
general case where $\scal$ has base points.

\begin{prop}\label{EZ} {\rm(\cite[Prop.~3.1]{SZ}, \cite[Prop.~2.1]{SZvar})}
Let $(L,h)$ be a Hermitian holomorphic line bundle on a complex
manifold
$M$, and let
$\scal$ be a finite-dimensional subspace of $ H^0(M,L)$.  We
give
$\scal$ an
inner product and we let $\ga$ be the induced  Gaussian probability
measure on $\scal$.
Then the expected zero current of a random section
$s\in\scal$ is given by

\begin{eqnarray*}\E_\ga(Z_s)  &=&\frac{\sqrt{-1}}{2\pi}
\partial
\bar{\partial} \log \Pi_{\scal}(z, z)+c_1(L,h)\;.\end{eqnarray*}
\end{prop}

We note that the expected zero current $\E_\ga(Z_s)$ is a smooth form outside
the base point set of $\scal$. Proposition \ref{EZ} also holds for infinite 
dimensional spaces $\scal$; see \cite{SZvar1}.

We next state our general result on simultaneous expected zeros:

\begin{prop}\label{EZsimult}  Let $M$ be a projective algebraic manifold, and let
$(L_1,h_1),\dots ,(L_k,h_k)$ be Hermitian holomorphic line bundles on $M$
($1\le k\le \dim M$). Suppose we are given subspaces
$\scal_j\subset H^0(M,L_j)$ with inner products
$\langle,\rangle_j$ and let $\ga_j$ denote the associated Gaussian
probability measure on $\scal_j$ (for $1\le j\le k$).  Let $U$ be
an open subset of  $M$ on which $\scal_j$ has no base points for
all $j$. Then the expected simultaneous zero current of independent random
sections $s_1\in\scal_1,\;\dots,\;s_k\in\scal_k$ is given over $U$
by
$$\E_{\ga_1\times\cdots\times\ga_k}(Z_{s_1,\dots,s_k}) = \bigwedge_{j=1}^k
\E_{\ga_j}(Z_{s_j}) = \bigwedge_{j=1}^k
\left[\frac{\sqrt{-1}}{2\pi}
\partial
\bar{\partial} \log \Pi_{\scal_j}(z, z)+c_1(L_j,h_j) \right]\;.$$
\end{prop}

Proposition \ref{EZsimult} is a generalization of Proposition 2.2 in \cite{SZvar}, 
where the formula is proved under the assumption that the $\scal_j$ are identical 
subspaces of the same line bundle and are base point free on all of $M$.  To use the 
argument  in  \cite{SZvar} to prove the above form of the proposition, we must first 
show that, for each fixed  test form $\phi\in \dcal^{m-k.m-k}_\R(U)$,  the map \begin
{equation}\label{map} (s_1,\dots,s_k) \mapsto (Z_{s_1,\dots,s_k}, \phi)\end{equation} 
is $L^\infty$.

 To verify this assertion, we let $A$ be a very ample line bundle of the form $A=L
\otimes L'$, where $L'$ is also very ample.  Suppose that the $Z_{s_j}$ are smooth 
divisors intersecting transversely in $U$, which is the case almost surely (by 
Bertini's theorem), since the $\scal_j$ have no base points in $U$.  Let $\wt s_j=s_j
\otimes t_j\in H^0(M,A)$,  where the sections $t_j\in H^0(M,L')$ are chosen so that 
the zero divisors $Z_{\wt s_j}$ are smooth and intersect transversely in $U$.  Next 
deform the sections $\wt s_j$ to sections $\sigma_j^\nu \in H^0(M,A)$, with  $
\sigma_j^\nu\to \wt s_j$ as $\nu\to\infty$, such that the zero divisors $Z_{\sigma_j^
\nu}$ are smooth and intersect transversely on all of $M$.  Then 
$$\left(Z_{\sigma_1^\nu,\dots,\sigma_k^\nu},\om^{m-k}\right)=
 \int_M c_1(A,h)^k \wedge \om^{m-k}\;.
$$  Letting $\nu\to\infty$, we conclude that
\begin{equation}\label{upper} \int_{Z_{s_1,\dots,s_k}\cap U} \om^{m-k} \le
\int_{ Z_{\wt s_1,\dots,\wt s_k}\cap U} \om^{m-k} =  
\lim_{\nu\to\infty}  \int_{Z_{\sigma_1^\nu,\dots,\sigma_k^\nu}\cap U}  \om^{m-k} \le  
\int_M c_1(A,h)^k \wedge \om^{m-k},
\end{equation}
and hence
$$\left|\big(Z_{s_1,\dots,s_k},\phi\big)\right|
\le \frac{\|\phi\|_\infty}{(m-k)!}  \int_{Z_{s_1,\dots,s_k}\cap U} \om^{m-k} \le \frac{\|\phi\|_\infty}{(m-k)!} 
\int_M c_1(A,h)^k \wedge \om^{m-k}
\;,$$ verifying that the map \eqref{map} is bounded.

We now can apply the proof in \cite{SZvar}: The case $k=1$ follows from Proposition 
\ref{EZ} with $M=U$, and the inductive step follows by the proof of Proposition 2.2 in 
\cite{SZvar} with $M$ replaced by $U$.\qed

\section{The variance estimate}\label{s-variance}
In this section, we prove the following variance estimate, which is a slight
generalization of Theorem
\ref{var-intro}.

\begin{theo}\label{mainvar} Let $L_1,\dots ,L_k$ be stably base point free
line bundles on  a projective \kahler manifold $(M,\om)$, where $1\le k\le
m=\dim M$. Suppose we are given subspaces $\scal_j\subset H^0(M,L_j)$
endowed with Gaussian probability measures $d\ga_j$ (for $1\le j\le k$). Let
$\phi\in \dcal^{m-k,m-k}_\R(U)$, where $U$ is an open subset of $M$ on which
$\scal_{j}$ has no base points for $1\le j\le k$.

Then for a system
$\s=(s_1,\dots,s_k)$ of sections of $\scal_1,\dots,\scal_k$ chosen
independently and at random, we have:
$$\sqrt{\var (Z_{{\s}},\phi)} \le C_m  \,
\|\ddbar \phi\|_{\infty}\int_M \om^{m-k+1}\wedge\sum_{\la=1}^k \left[\prod_{1\le j\le 
k,\,j\neq\la}c_1
(L_j)\right]\;,$$ where
$C_m$ is a universal constant depending only on the dimension $m$.\end{theo}

 A line bundle $L$ is said to be
{\it stably base point free\/} if the base point set of
$H^0(M,L^N)$ is empty for $N$ sufficiently large. In particular,
ample line bundles are stably base point free as a consequence
of the Kodaira embedding theorem.   

\begin{rem} We remark that the hypothesis
that $L$ is stably base point free is essential for the estimate
of Theorem \ref{mainvar}; indeed, the stated upper bound of the theorem might be 
negative. For
example, let $m=k=3$ and let $M=Y\times \CP^1$, where $Y$ is the
blow-up of a point in $\CP^2$. To construct the line bundles,
we let $E$ be the exceptional divisor in $Y$ and $\wh H$ the
pull-back to $Y$ of a line $H\subset\CP^2$, and we let $\pi_1:M\to Y,\
\pi_2:M\to\CP^1$ denote the projections. We then let $L_j=L_D$ ($j=1,2,3$),
where $D=\pi_1^*(\wh H+4E)+\pi_2^*\{p\}$, and let
$[\om]=\pi_1^*(2\wh H-E)+\pi_2^*\{p\}$.  Then $$\int_M \om\wedge
c_1(L_D)^2=(\wh H+4E)^2 + 2\,(\wh H+4E)\cdot(2\wh H-E) = -15 +12=-3\;.$$
\end{rem}

We shall prove Theorem \ref{mainvar} by induction on $k$. The case $k = 1$
is  essentially Lemma 3.3 in \cite{SZ}. To go from  $k = 1$ to $k = 2$ (and
subsequently to higher $k$), we shall use the fact that  $Z_{s_1, s_2} = Z_{s_1}
\wedge Z_{s_2}$ is the current  of integration over the intersection $Z_{s_1}
\cap Z_{s_2}$  and  hence $\left(Z_{s_1, s_2}, \phi
\right)$ reduces to the integration of $\phi|_{Z_{s_1}}$ against
$Z_{ s_2}|_{Z_{s_1}\cap U}$, which is almost surely smooth.

We begin with the $k = 1$ step, which is based on a result from \cite{SZ}.

\begin{lem} \label{variance}
Under the hypotheses and notation of Proposition \ref{EZ}, we have
$$\var( Z_{s},\phi) \leq\;  C  \|\ddbar \phi\|_1^2\;,\quad \phi\in
\dcal^{m-1,m-1}_\R(M)\;,$$ where $C$ is a universal constant. (The main point
is that the constant $C$ is independent of $\dim \scal$ as well as $M$ and
$L$.)
\end{lem}

\begin{proof}   For completeness, we include a modified version of
the argument of \cite[Lemma~3.3]{SZ}.  As in \S\ref{EDZ}, we let $\{S_j\}$ be an 
orthonormal basis for $\scal$ and 
we write sections locally as $$s=\sum_{j=1}^n c_jS_j = \langle
c,\s \rangle = \langle c, F\rangle e_L$$ where $c=(c_1,\dots,c_n)$, $\s = (S_1,
\dots,S_n)$, $F=(f_1,\dots,f_n)$. By \eqref{chern}--\eqref{Zs}, we have
\begin{equation*} Z_s =
\frac{\sqrt{-1}}{ \pi } \partial \bar{\partial}\log | \langle c,F
\rangle| = \frac{\sqrt{-1}}{ \pi } \partial \bar{\partial}\log| \langle c,\s
\rangle|_h + c_1(L,h)\;. \end{equation*}

Let $\phi\in
\dcal^{m-1,m-1}_\R(M)$ and consider the random variable $Y:\scal\to\C$ given by \begin
{eqnarray}Y(s)&=&(Z_s,\phi) - \int_M c_1(L,h)\wedge\phi \nonumber\\& =&  \left(\frac
{\sqrt{-1}}{ \pi } \partial \bar{\partial}\log| \langle c,\s
\rangle|_h,\,\phi\right)\nonumber\\& =& \frac{\sqrt{-1}}{ \pi }\int_M\log| \langle c,
\s
\rangle|_h  \;\ddbar\phi\;.\label{X}\end{eqnarray} 
We note that $\var(Z_s,\phi)=\var(Y)$.  

By Proposition \ref{EZ}, we have \begin{equation}\label{EZX} \E(Y)= \frac{\sqrt{-1}}{2
\pi}\int_M\log\Pi_\scal(z,z)\;  \ddbar\phi(z) = \frac{\sqrt{-1}}{\pi}\int_M \log|\s|_h
\;\ddbar\phi\;.\end{equation}
Furthermore, by \eqref{X} we have \begin{equation}\label{EZV} \E(Y^2)=
\frac{-1}{\pi^2} \int_M \int_M \ddbar \phi(z)\;\ddbar \phi(w)
\int_{\C^n} \log |\langle c,\s(z)\rangle|_h\, \log |\langle c,\s(w)
\rangle|_h\, d\ga(c)\;.
\end{equation} We let $u(z) = |\s(z)|_h\inv \s(z)$ so that $|u(z)|_h\equiv
1$, and we have
\begin{multline*}\log |\langle c,\s(z)\rangle|_h \,\log |\langle c,\s(w)\rangle|_h =
\log |\s(z)|_h \,\log |\s(w)|_h + \log|\s(z)|_h \,\log |\langle c,u(w)\rangle|_h 
\\+ \log |\langle c,u(z)\rangle |_h\, \log |\s(w)|_h  +
 \log |\langle c,u(z)\rangle |_h  \,\log |\langle c, u(w)\rangle
|_h\;,\end{multline*} which decomposes (\ref{EZV}) into four terms. By \eqref{EZX}, 
the
first term contributes
\begin{equation}\frac{-1}{\pi^2}  \int_M \int_M \ddbar \phi(z)\;\ddbar 
\phi(w)\log |\s(z)|_h \,\log |\s(w)|_h  =  (\E Y )^2\;.
\end{equation} The $c$-integral of the second term is independent of $w$ and hence
the second term in the expansion of \eqref{EZV} vanishes.  The third term likewise 
vanishes. Therefore,
\begin{equation} \label{varint} \var( Z_{s},\phi) = \frac{-1}{\pi^2}
 \int_M \int_M \ddbar \phi(z)\;\ddbar \phi(w)
\int_{\C^n} \log |\langle c,u(z)\rangle|_h \,\log |\langle c,u(w)\rangle|_h\, d\ga(c)
\;.\end{equation} By Cauchy-Schwartz,
\begin{eqnarray}&&\hspace{-.5in}\left|\int_{\C^n} \log |\langle c,u(z)\rangle| \,\log 
|\langle
c,u(w)\rangle| d\ga(c)\right|\nonumber \\&& \le
\left(\int_{\C^n} (\log |\langle c,u(z)\rangle|)^2
 d\ga(c)\right)^\half\left(\int_{\C^n}  (\log
|\langle c,u(w)\rangle|)^2 d\ga(c)\right)^\half\nonumber\\&&
=\int_{\C^n} (\log |c_1|)^2 d\ga(c)\ =\ \frac 1 \pi \int_\C (\log
|c_1|)^2e^{-|c_1|^2}\, dc_1\;.\label{logint}\end{eqnarray} The
conclusion follows immediately from
(\ref{varint})--(\ref{logint}).
\end{proof}

\medskip\noindent{\it Proof of Theorem \ref{mainvar}:\/}
We shall prove by induction on $k$ that the variance bound holds when $\omega$ is
an arbitrary closed semi-positive $(1,1)$-form  on $M$ that is strictly
positive on $U$.  Let $\om$ be such a form, and let $\Om:=\frac 1{m!}\om^m|U$ denote 
the induced volume form on $U$.  Let 
$\eta\in\dcal^{2m}(U)$ be a compactly supported, top degree form on $U$, and write
$\eta=f\Om$.
We define the sup norm $\|\eta\|_{\infty}:=
\|f\|_{\infty}$.  The $L^1$ norm is given by
\begin{equation}\label{L1}\|\eta\|_1 = \int_U |f|\Om \le 
\|\eta\|_{\infty}\int_U\Om = \|\eta\|_\infty \vol(U)\;.\end{equation}
We note that while the $L^\infty$ norm depends on $\om$, the $L^1$ norm of $\eta$ is
independent of the choice of the \kahler form  on $U$.

The case $k=1$ is an immediate consequence of Lemma \ref{variance} (with $M$ replaced 
by $U$) and \eqref{L1}. Now let
$2\le k\le m$ and assume the inequality has been proven for $k-1$ sections.
We let $\s=(s_1,\dots,s_k)\in \prod_{j=1}^k \scal_j$ be a random $k$-tuple of
sections.  We write
$\s=(\s',s_k)$, where
$\s'=(s_1,\dots,s_{k-1})$. 

By  Bertini's Theorem,
the hypersurfaces $|Z_{s_j}|$ ($1\le j\le k$) are smooth in $U$
and intersect transversely in $U$ for almost all $\s$, so that we may write
$Z_{\s} = Z_{\s'}\wedge Z_{s_k}$. Let $\phi\in\dcal^{m-k,m-k}_\R(U)$ be a test form.  
By Proposition \ref{EZsimult}, $\E\wt
Z_{\s} = \E Z_{\s'}\wedge \E Z_{s_k}$ and hence
\begin{equation}\label{vark} \var (Z_\s,\phi)= \E(Z_{\s'}\wedge Z_{s_k},\phi)^2 - (\E
Z_{\s'}\wedge \E Z_{s_k},\phi)^2\;.\end{equation} We write \begin{eqnarray} (Z_{\s'}
\wedge Z_{s_k},\phi)^2 - (\E
Z_{\s'}\wedge \E Z_{s_k},\phi)^2 &=& G_1 +G_2\;,\label{G1G2} \qquad \mbox{where}\\G_1\ 
= \
G_1(\s',s_k)&=&(Z_{\s'}\wedge Z_{s_k},\phi)^2 -(Z_{\s'}\wedge \E Z_{s_k},\phi)^2\quad 
a.e.\,,\label{part1}\\G_2\ = \ G_2(\s')&=&
(Z_{\s'}\wedge \E Z_{s_k},\phi)^2- (\E Z_{\s'}\wedge \E Z_{s_k},\phi)^2\ 
a.e.\label{part2}\end{eqnarray}
Hence \begin{equation}\label{varG1G2} \var (Z_\s,\phi)= \E G_1+\E G_2\;.\end{equation}

We now let $V=|Z_{\s'}|\cap U$. (Recall that $|Z|$ denotes the
support of a zero current $Z$.) The idea of the proof is to notice
that $(Z_{\s'}\wedge Z_{s_k},\phi)=( Z_{s_k}|_{V},\phi|_{V})$ and then
to apply Lemma  \ref{variance} with $M$ replaced by $|Z_{\s'}|$ and
with $Z_s= Z_{s_k}$ in order to obtain the desired bound
for $\E G_1$. We then reverse the roles of $Z_{\s'}$ and $Z_{s_k}$ and use
a similar argument to obtain the bound for $\E G_2$.

To obtain a bound for $\E G_1$, we first integrate over
$\scal_k$:
\begin{eqnarray} \int_{\scal_k} G_1(\s',s_k)\,d\ga_k(s_k)
&=& \int_{\scal_k}\left[( Z_{s_k}|_{V},\phi|_{V})^2 -(\E
Z_{s_k}|_{V},\phi|_{V})^2\right]\,d\ga_k(s_k)\nonumber\\ &=&
\var(Z_{s_k}|_{V},\phi|_{V}) \ \le \ C\left(
\int_{V}|\ddbar \phi|\right)^2\nonumber\\
&\le & C\left[\|\ddbar\phi\|^2_\infty
 \int_{V}\om^{m-k+1}\right]^2\;,
\label{ineqs}\end{eqnarray} where the first inequality is by Lemma
\ref{variance} (with $M$ replaced by $V$).

We claim that
\begin{equation}\label{claim1} \int_{V}\om^{m-k+1} \le \int_M
 \om^{m-k+1}\wedge\prod_{j=1}^{k-1}c_1(L_j)\quad \mbox {for
a.a. } \s'\;.\end{equation} To verify (\ref{claim1}), choose a positive integer $N$ so 
that the
line bundles $L_j^N$ are base point free, and then (by applying
Bertini's theorem, as before) deform the sections
$s_j^{\otimes N}$ to sections $\sigma_j^\nu\in H^0(M,L_j^N)$, with $\sigma_j^\nu\to 
s_j^{\otimes N}$ as $\nu\to\infty$, such that the zero sets
$Z_{\sigma_1^\nu,\dots,\sigma_{k-1}^\nu}$ are smooth
reduced varieties of  dimension $m-k+1$ (in all of $M$). We then have
\begin{equation}\label{ineq1}\int_{Z_{\sigma_1^\nu,\dots,\sigma_{k-1}^\nu}\cap U}\om^{m-k+1} \le \int_{Z_{\sigma_1^\nu,\dots,\sigma_{k-1}^\nu}}\om^{m-k+1}
= N^{k-1}\int_M
\om^{m-k+1}\wedge\prod_{j=1}^{k-1}c_1(L_j)\;.\end{equation}
Letting $\nu\to\infty$ and noting that $$\int_{Z_{\sigma_1^\nu,\dots,\sigma_{k-1}^\nu}\cap U}\om^{m-k+1} \to N^{k-1}\int_
V\om^{m-k+1}\;,$$ we obtain (\ref{claim1}).
Hence by (\ref{ineqs})--(\ref{claim1}), we have
\begin{equation}\label{EG1} \E G_1 = \int_{\scal'}
\int_{\scal_k} G_1(\s',s_k)\,d\ga_k(s_k)\, d\ga'(\s') \le C\,
\|\ddbar \phi\|_\infty^2\,\left( \int_M
\om^{m-k+1}\wedge\prod_{j=1}^{k-1}c_1(L_j)\right)^2\;,
\end{equation}
where $\scal'=\scal_1\times \cdots\times\scal_{k-1}$, $\ga'=\ga_1\times\cdots\times \ga_{k-1}$, and $C$ 
denotes a constant depending only on $m$.

We now estimate $\E G_2$. First we note that $\E(G_2)$ is the variance of the   random 
variable $X$ on $\scal'$ given a.e.\ by
$$X(\s') = \left( Z_{\s'} \wedge \E Z_{s_k},\phi \right).$$
Hence $$\E G_2 = \var (X)=\E G_2'\;,$$ where
\begin{equation}\label{part2*}  G_2'(\s') :=\big[X(\s')-\E X\big]^2 =
\big((Z_{\s'} - \E Z_{\s'})\wedge \E  Z_{s_k},\phi \big)^2\;.\end{equation}
By Cauchy-Schwartz, we have the upper bound
\begin{eqnarray*} G'_2(\s') &=& \left[\int_{\scal_k}
\big((Z_{\s'} - \E Z_{\s'})\wedge  Z_{s_k},\phi \big)
\,d\ga_k(s_k)\right]^2\\ & \le &
\int_{\scal_k} \big((Z_{\s'} - \E Z_{\s'})\wedge  Z_{s_k},\phi \big)^2
\,d\ga_k(s_k)\;,
\end{eqnarray*} for all $\s'$ such that $Z_{\s'}\cap U$ is a smooth submanifold of 
codimension $k-1$.

Writing $W_{s_k} = | Z_{s_k}|\cap U$, we then have  \begin{eqnarray}  \E G_2 \ =\ \E 
G'_2 & \le  &\int_{\scal'}d\ga'(\s')\ \int_{\scal_k}d\ga_k(s_k)\ 
\big((Z_{\s'} - \E Z_{\s'})\wedge  Z_{s_k},\phi \big)^2
\nonumber \\ &=& 
\int_{\scal_k}d\ga_k(s_k)\
\int_{\scal'}d\ga'(\s')\ \big((Z_{\s'} - \E
Z_{\s'})|_{W_{s_k}},\phi|_{W_{s_k}} \big)^2 \nonumber\\ & = & \int_{\scal_k}d\ga_k
(s_k)\
\var ( Z_{\s'}|_{W_{s_k}},\phi|_{W_{s_k}} )\;,\label{EG2}
 \end{eqnarray} where the variance is with respect to $\s'$.
If $|Z_{s_k}|$ is a smooth submanifold of $M$, then we can apply the inductive 
hypothesis
to $|Z_{s_k}|$ to conclude that
\begin{eqnarray}\label{ifsmooth} \sqrt{\var ( Z_{\s'}|_{W_{s_k}},\phi|_{W_{s_k}} )} & 
\le &
C_{m-1} \,\|\ddbar\phi\|_\infty\int_{|Z_{s_k}|}\om^{m-k+1} \wedge
\sum_{\la=1}^{k-1}
 \left[\prod_{1\le j\le k-1,\,j\neq\la}c_1
(L_j)\right]  \nonumber\\ &=& C_{m-1} \,\|\ddbar\phi\|_\infty\int_M\om^{m-k+1} \wedge
\sum_{\la=1}^{k-1} \left[\prod_{1\le j\le k,\,j\neq\la}c_1
(L_j)\right]\;.\end{eqnarray} 

If $L_k$ is base point free on $M$, then $|Z_{s_k}|$ will almost surely be smooth and 
hence \eqref{ifsmooth} will hold almost surely.  For the general case,  we use the 
following argument:  Since $\scal_k$ has no base points in $U$, $Z_{s_k}$  will almost
surely be smooth in $U$. Now suppose $Z_{s_k}$ is smooth in $U$, but has
singularities in $M$.  Let $\pi:\wt M\to M$ be a resolution of the
singularities of
$Z_{s_k}$; i.e., $\pi$ is a modification of
$M$ that is biholomorphic outside the singular locus of $Z_{s_k}$ such that
the proper transform
$\wt Z_{s_k}\subset
\wt M$ of
$Z_{s_k}$ is smooth. Since $Z_{s_k}$ is smooth in $U$, $\pi$ does not blow
up points of $U$.  Applying the inductive assumption to the linear systems
$\wt \scal_j:= \pi^*\scal_j|\wt Z_{s_k}$ ($1\le j\le k-1$) and semi-positive form $\wt
\om := \pi^*\om|\wt Z_{s_k}$ (which is strictly positive on $\wt U=U$), we obtain
(\ref{ifsmooth}) for almost all
$s_k\in\scal_k$.

Hence it follows from (\ref{EG2})--(\ref{ifsmooth}) that
\begin{equation}\E G_2 \le C_{m-1}^2\,
\|\ddbar \phi\|_\infty^2\,\left( \int_M\om^{m-k+1} \wedge
\sum_{\la=1}^{k-1} \left[\prod_{1\le j\le k,\,j\neq\la}c_1
(L_j)\right]\right)^2\;,
\label{EG2*}\end{equation}
The inductive step follows from (\ref{varG1G2}), (\ref{EG1}) and
(\ref{EG2*}).
\qed

\section{Almost sure convergence of zeros}\label{AE}
We complete this paper by verifying Theorem \ref{varN} and
Corollary \ref{convergence}.  Theorem \ref{varN} is a consequence of the following 
variant of Theorem \ref{var-intro}: 

\begin{prop}\label{var-var} Let $L\to (M,\om)$ be a
holomorphic line bundle  over a compact \kahler manifold of
dimension $m$, and let $\scal_j\subset H^0(M,L)$, $j=1,\dots,k$, be linear spaces 
endowed with Gaussian probability measures. Let  $\phi\in
\dcal^{m-k,m-k}_\R(U)$,  where $U$ is an open subset of $M$ on which
$\scal_{j}$ has no base points for $j=1,\dots, k$.

Suppose that $A$ is a very ample line bundle on $M$ of the form $A=L\otimes L'$, where $L'$ 
is also very ample. Then for independent random sections $s_j\in\scal_j$, 
$$\sqrt{\var ( Z_{s_1,\dots,s_k},\phi)} \le C_m\,
\|\ddbar \phi\|_{\infty}\int_M \om^{m-k+1}\wedge c_1(A)^{k-1}\;,$$
where the constant
$C_m$ depends only on the dimension $m$ of $M$.
\end{prop}

\begin{proof}  The result follows by repeating the proof of Theorem \ref{mainvar} with $c_1(L_j)$ replaced by $c_1(A)$.  
Instead of \eqref{claim1}, we use the inequality
\begin{equation}\label{claim2} \int_{|Z_{\s'}|\cap U}\om^{m-k+1} \le \int_M
 \om^{m-k+1}\wedge c_1(A)^{k-1}\;,\end{equation} where  $\s'=(s_1,\dots,s_{k-1})\in 
\prod_{j=1}^{k-1}\scal_j$ is chosen as before so that $|Z_{\s'}|\cap U$ is a smooth
reduced submanifold of  dimension $m-k+1$.  The inequality \eqref{claim2} is the 
same as the inequality \eqref{upper} (with $k$ replaced by $k-1$), which was verified in the proof of Proposition \ref{EZsimult}.  In place of \eqref{ifsmooth}, we have
\begin{eqnarray}\label{ifsmooth2} \sqrt{\var ( Z_{\s'}|_{W_{s_k}},\phi|_{W_{s_k}} )} & 
\le &
C_{m-1} \,\|\ddbar\phi\|_\infty\int_{|Z_{s_k}|}\om^{m-k+1} \wedge
c_1(A)^{k-1}  \nonumber\\ &=& C_{m-1} \,\|\ddbar\phi\|_\infty\int_M\om^{m-k+1} \wedge
c_1(A)^{k-1} \wedge c_1(L) \nonumber\\ & \le & C_{m-1} \,\|\ddbar\phi\|_\infty\int_M\om^{m-k+1} \wedge
c_1(A)^k \;.\end{eqnarray} 
The proof of \eqref{ifsmooth2} is exactly the same as that of \eqref{ifsmooth}.
\end{proof}

\noindent{\it Proof of Theorem  \ref{varN}:\/} Let $L\to M,\ U,\ (\scal^j_N, \ga^j_N),
\ \phi$ be as in the statement of the theorem.  Let $L'$ be an ample line bundle on $M
$ such that the line bundle $A:=L\otimes L'$ is ample.
Applying  Proposition \ref{var-var} with $L,L',A$ replaced with $L^N,L'^N,A^N$, 
respectively, we conclude that
$$\sqrt{\var ( Z_{s^1_N,\dots,s^k_N},\phi)} \le C_m\,
\|\ddbar \phi\|_{\infty}\,N^{k-1}\int_M \om^{m-k+1}\wedge c_1(A)^{k-1} = O(N^{k-1})\;,
$$ and hence
$$\sqrt{\var (\wt Z_{s^1_N,\dots,s^k_N},\phi)} = O(N\inv)\;.$$
\qed

\bigskip\noindent{\it Proof of Corollary  \ref{convergence}:\/} The proof follows from the 
elementary argument in \S 3.3 of \cite{SZ}, which we include here for completeness.  Consider a random 
sequence
${\bf s}=\{\s_N\}$  in
$\scal_\infty$, where $\s_N=(s^1_N,\dots,s^k_N)\in (\scal_N)^k$.  Since the masses of
$\wt Z_{\s_N}$ are bounded independent of $N$, we may assume that
$\phi$ is a smooth form in $\dcal^{m-k,m-k}_\R(U)$. Now
consider the random variables
\begin{equation} Y_N({\bf s}): =(\wt Z_{\s_N}- \E\wt
Z_{\s_N},\phi)^2\geq 0\;.\end{equation}
 By Theorem~\ref{varN}, we have
$$\int_{{\mathcal S}} Y_N({\bf s}) d\ga({\bf s})
 = \var(\wt Z_{\s_N},\phi)=
 O\left(\frac{1}{N^{2}}\right)\;.$$ Therefore
$$\int_{{\mathcal S}}\sum_{N=1}^\infty Y_N d\ga =
\sum_{N=1}^\infty \int_{\mathcal S} Y_N d\ga <+\infty,$$ and hence
$Y_N\to 0$ almost surely, i.e.
\begin{equation}\label{conv-as} (\wt Z_{\s_N},\phi)-(
\E\wt Z_{\s_N},\phi) \to 0 \qquad a.s.\end{equation}
By hypothesis,
\begin{equation}(\E\wt Z_{\s_N},\phi) \to
\int_U \Psi\wedge\phi\;,\label{by4}\end{equation} and therefore by
(\ref{conv-as})--(\ref{by4}),
$$(\wt Z_{\s_N},\phi)\to
\int_U \Psi\wedge\phi
 \qquad a.s.\;,$$
completing the proof of Corollary \ref{convergence} \qed


\begin{thebibliography}{WWW}


\bibitem [Be1]{Be1} R. Berman, Bergman kernels, random zeroes and equilibrium 
measures for polarized pseudoconcave domains, arXiv:math/0608226v2.

\bibitem [Be2]{Be2} R. Berman, Bergman kernels and equilibrium measures for 
ample line bundles, 	arXiv:0704.1640v1.

\bibitem[BSZ1]{BSZuniv} P. Bleher, B. Shiffman and S. Zelditch, Universality
and scaling of correlations between zeros on complex manifolds,  {\it
Invent.\
Math.} 142 (2000), 351--395.

\bibitem[BSZ2]{BSZsusy} P. Bleher, B. Shiffman, and S. Zelditch,  Correlations
between zeros and supersymmetry, {\it Comm.\ Math.\ Phys.} 224 (2001),
255--269.

\bibitem[Bl1]{Bl} T. Bloom, Random polynomials and Green functions,
{\it  Int.\ Math.\ Res.\ Not.}  2005 (2005), 1689--1708.

\bibitem[Bl2]{Bl2} T. Bloom, Random polynomials and (pluri)potential 
theory, {\it Ann.\ Polon.\ Math.} 91 (2007), 131-141.


\bibitem[BS]{BS} T. Bloom and B. Shiffman, Zeros of random polynomials on
$\C^m$, {\it Math.\ Res.\ Lett.} 14 (2007), 469Ð-479.

\bibitem[BBL]{BBL} E.  Bogomolny, O.  Bohigas, and P.  Leboeuf, Quantum
chaotic dynamics and random polynomials, {\it J.  Statist.\ Phys.} 85 (1996),
639--679.

\bibitem[Ca]{Ca} D. Catlin, The Bergman kernel and a theorem of Tian,
{\it Analysis and Geometry in Several Complex Variables\/}, G. Komatsu
and M. Kuranishi, eds., Birkh\"auser, Boston, MA, 1999.

\bibitem[EK]{EK} A.  Edelman and E.   Kostlan, How many zeros of a random
polynomial
are real? {\it Bull.\ Amer.\ Math.\ Soc.} 32 (1995), 1--37.


\bibitem[FH]{FH} P. J. Forrester and G. Honner, Exact statistical
properties
of the zeros of complex random polynomials, {\it J. Phys.\ A} 32  (1999),
2961--2981.

\bibitem[Ham]{Ham} J. M. Hammersley, The zeros of a random polynomial, {\it
Proceedings
of the Third Berkeley Symposium on Mathematical Statistics and Probability,
1954--1955\/}, vol.\ II,  89--111, University of California Press, Berkeley
and Los
Angeles, 1956.

\bibitem[Han]{Han} J. H. Hannay, Chaotic analytic zero points: exact
statistics for those
of a random spin state, {\it J. Phys.\ A} 29 (1996), L101--L105.


\bibitem[Kac]{Kac} M. Kac, On the average number of real roots of a random
algebraic
equation, II, {\it Proc.\ London Math. Soc.} 50 (1949), 390--408.

\bibitem[NV]{NV} S.  Nonnenmacher and A.  Voros, Chaotic eigenfunctions in
phase space,  {\it J.
Statist.\ Phys.} 92 (1998), 431--518.

\bibitem[Ro]{Ro} J. M. Rojas,
On the average number of real roots of certain random sparse polynomial systems, 
{\it
The mathematics of numerical analysis (Park City, UT, 1995)\/}, 689--699, 
Lectures in Appl.\ Math. 32, Amer.\ Math.\ Soc., Providence, RI, 1996.

\bibitem[SZ1]{SZ} B. Shiffman and S. Zelditch, Distribution of zeros of
random and quantum chaotic sections of positive line bundles,
{\it Comm.\ Math.\ Phys.} 200 (1999), 661--683.


\bibitem[SZ2]{SZequil} B. Shiffman and S. Zelditch,  Equilibrium distribution
of zeros of random polynomials, {\it Int.\ Math.\ Res.\ Not.} 2003
(2003), 25--49.

\bibitem[SZ3]{SZnewton} B. Shiffman and S. Zelditch, Random polynomials with
prescribed Newton polytope, {\it J. Amer.\ Math.\ Soc.} 17 (2004), 49--108.

\bibitem[SZ4]{SZvar} B. Shiffman and S. Zelditch,
Number variance of random zeros on complex manifolds, {\it Geom.\ Funct.\ Anal.}, to appear.

\bibitem[SZ4a]{SZvar1} B. Shiffman and S. Zelditch,
Number variance of random zeros on complex manifolds, arXiv:math/0608743v1. 
(This  early version of \cite{SZvar} contains additional results to be published elsewhere.)

\bibitem[SS]{SS} M. Shub and S. Smale, 
Complexity of Bezout's theorem. II. Volumes and probabilities,  
{\it Computational algebraic geometry (Nice, 1992)\/}, 267--285, 
Progr.\ Math.\ 109, Birkh\"auser, Boston, MA, 1993.

\bibitem[ST]{ST} M. Sodin and B. Tsirelson,
Random complex zeros, I. Asymptotic normality, {\it Israel J. Math.} 144
(2004), 125--149.

\bibitem[Ti]{Ti} G.  Tian, On a set of polarized \kahler metrics on algebraic
manifolds, {\it J.  Diff.\ Geometry\/} 32 (1990), 99--130.


\bibitem[Ws]{Ws} M. Wschebor, 
On the Kostlan-Shub-Smale model for random polynomial systems. Variance of the number
of roots, {\it
J. Complexity\/} 21 (2005), 773--789.

\bibitem[Ze]{Ze} S. Zelditch, \szego kernels and a theorem of Tian,
{\it Int.\ Math.\ Res.\ Not.} 1998 (1998),  317--331.

\bibitem[Zr]{Zr} S. Zrebiec,  The order of the decay of the hole probability 
for Gaussian random $\SU(m+1)$ polynomials, arXiv:0704.2733v1.
 

\end{thebibliography}
\end{document}